%% file: Diallo2013prescribing.tex
\documentclass[12 pt,english]{amsart}

\usepackage{babel}
\usepackage{amsmath}
\usepackage{amsfonts}
\usepackage{latexsym}
\usepackage{color}
\usepackage{graphicx}

\begin{document}

\newtheorem{cor}{Corollary}
\newtheorem{lem}[cor]{Lemma}
\newtheorem{pro}[cor]{Proposition}
\newtheorem{bigtheo}[cor]{Main Theorem}
\newtheorem{theo}[cor]{Theorem}
\newtheorem{sub}[cor]{Sublemma}

\title[Prescribing metrics]{Prescribing metrics on the boundary of convex cores of globally hyperbolic maximal compact AdS $3$-manifolds}
\author{Boubacar DIALLO}
\address{Institut de Math\'ematiques de Toulouse, UMR CNRS 5219, Universit\'e Toulouse 3, 118 route de Narbonne, 31062 Toulouse cedex 9, France}
\thanks{partially supported by the A.N.R. through project
ETTT, ANR-09-BLAN-0116-01, 2009-13.}
\email{boubacar.diallo.math@gmail.com}
\date{\today}

\begin{abstract}{We consider globally hyperbolic maximal anti de Sitter $3$-manifolds $M$ with a closed Cauchy surface $S$ of genus greater than one and prove that any pair of hyperbolic metrics on $S$ can be realized as the boundary metrics of the convex core of a maximal globally hyperbolic anti de Sitter $3$-manifold structure on $M$. This answers the existence part of a question of Mess about the unique realization of such metrics. Our theorem has a nice formulation purely in terms of $2$-dimensional Teichm\"uller theory and earthquakes.}
\end{abstract}

\maketitle

\tableofcontents

\section{Introduction}

\subsection{Globally hyperbolic maximal compact Anti de Sitter manifolds}

Anti de Sitter space (AdS) is the model space of Lorentzian manifolds of constant sectional curvature equal to $-1$. Its $n-$dimensional avatar can be defined as the $-1$ level set of the standard nondegenerate quadratic form of index $2$ in $\mathbb{R}^{n+1}$, with the induced quadratic form on each affine tangent subspace (it has index one on such subspaces).

An anti de Sitter $3$-manifold is thus a (smooth, connected) manifold $M$ endowed with a symmetric bilinear covariant $2$-tensor of index one, everywhere nondegenerate, whose sectional curvatures are all equal to $-1$. By classical results, such an $M$ is locally isometric to the model space $AdS_{3}$, anti de Sitter space of dimension 3. If we restrict ourselves to oriented and time oriented manifolds, it is therefore endowed with a $(G,X)$ structure, where $X$ is $AdS_{3}$ and $G$ the identity component of its isometry group. Both definition are equivalent in that case. Let's call such objects $AdS_{3}$ spacetimes.

An $AdS_{3}$ spacetime is said to be globally hyperbolic if it admits a Cauchy hypersurface: a spacelike surface which intersect every inextendable timelike line exactly once. If the spacetime has a compact Cauchy surface, then every Cauchy surface is compact. Moreover if the spacetime cannot be isometrically embedded in a stricty larger spacetime by an isometry sending a Cauchy surface to another one, then it is called globally hyperbolic maximal compact (GHMC).

$(G,X)$ manifolds (with $G$ and $X$ as above) have a well defined holonomy representation, up to conjugation, and a developing map (up to equivalence).

Recall that the identity component of the isometry group of $AdS_{3}$ identifies (up to index $2$) with $PSL_{2}(\mathbb{R}) \times PSL_{2}(\mathbb{R})$. Thus holonomies of $(G,X)$ spacetimes $M$ are represented by pairs of representations from the fundamental group of $M$ to $PSL_{2}(\mathbb{R})$. In the case where $M$ is globally hyperbolic maximal with a surface $S$ of negative Euler characteristic as a Cauchy surface, one of Mess's main theorem (~\cite[proposition 19]{MR2328921}, see also ~\cite{MR2328922}) asserts that such a pair ($\rho_{l}$, $\rho_{r}$) is a point of ${\mathcal T}(S) \times {\mathcal T}(S)$, the product of two copies of the Teichm\"uller space of $S$. Recall that ${\mathcal T}(S)$ is the space of discrete and faithful representations of $\Pi_{1}(S,x)$ (with any choice of basepoint $x$) to $PSL_{2}(\mathbb{R})$ modulo conjugation. Equivalently, it is the space of metrics on $S$ of constant sectional curvature $-1$, modulo the equivalence relation identifying two of them iff there is a diffeomorphism isotopic to the identity, which is an isometry from one metric to the other.

\subsection{Analogies with quasifuchsian hyperbolic $3$-manifolds}
Mess noted that his theorem on holonomies of $GHMC$ anti de Sitter spacetimes of dimension $3$ is the analog of the simultaneous uniformization theorem of Bers for quasifuchsian hyperbolic three manifolds ~\cite{bers}. Thurston asked whether one could uniquely prescribe the two hyperbolic metrics on the boundary of the convex core of such a manifold ~\cite{thurston-notes},~\cite{MR2235710}. So far, only existence has been proved, thanks to works of Epstein and Marden ~\cite{MR903852} on Thurston and Sullivan's $K=2$ conjecture (which happens to be false, see ~\cite{MR2153403}), and Labourie ~\cite{MR1163450}, independently.

In his work ~\cite{MR2328921}, Mess established further analogies between quasifuchsian hyperbolic $3$-manifolds and $GHMC$ $AdS_{3}$ spacetimes. Indeed such spacetimes $M$ have a well defined convex core, which as in the hyperbolic setting is the minimal non-empty closed convex subset. Except for the Fuchsian case where both upper and lower boundary metrics are equal, the convex core has two boundary components which are pleated hyperbolic surfaces. Both are thus equipped with hyperbolic metrics and measured bending laminations. In particular this defines a map $\Phi$ from the space of $GHMC$ $AdS$ structures (with Cauchy surface $S$ of fixed topological type), identified with ${\mathcal T}(S)\times {\mathcal T}(S)$, to ${\mathcal T}(S)\times {\mathcal T}(S)$, which sends a structure to the ordered pair of upper and lower boundary hyperbolic metrics. In the quasifuchsian setting, the analogous map is then onto.

\subsection{The Mess conjecture}

Mess asked whether the map $\Phi$ is one-to-one and onto, that is, whether an ordered pair of hyperbolic metrics on $S$ can be uniquely realized as the upper and lower boundary metrics of a $GHMC$ $AdS_{3}$ spacetime $M$. This is the analog of Thurston's conjecture for quasifuchsian manifolds. Uniqueness is still an open question. The present article gives a positive answer to the existence part of this conjecture of Mess.

Note that our statement, in the anti de Sitter setting, cannot be proved by methods of Epstein and Marden. Indeed, the $K$-quasiconformal constant in their theorem cannot exist in our context, because it would contradict the earthquake theorem (see next section). There's no restriction on the bending measures of our spacetimes, as opposed to the hyperbolic setting. Moreover, the analog of Labourie's theorem ~\cite[theor\`eme 1]{MR1163450} or of Labourie and Schlenker's theorem ~\cite{MR1752780} remains unknown in the $AdS$ setting.

\subsection{Relation to Teichm\"uller theory and earthquakes}

Earthquakes were defined by Thurston in ~\cite{MR903860} by extension to measured laminations of the case of simple closed curved.
Let $E_{\lambda}^{l}$ and $E_{\lambda}^{r}$ be the left and right earthquakes along a measured lamination $\lambda$. Thurston proved that those two maps from ${\mathcal T}(S)$ to itself are (continuous and) bijective and in fact are inverse to each other. His earthquake theorem asserts that for any $m$ and $m'$ in ${\mathcal T}(S)$, there is a unique measured lamination $\lambda$ (resp. $\lambda'$) such that the left (resp. right) earthquake along $\lambda$ (resp. $\lambda'$) sends $m$ to $m'$.

Mess rephrased this earthquake theorem in the context of pleated surfaces in $AdS_{3}$ geometry.

Let $S$ be a closed surface of negative Euler characteristic, $M$ a globally hyperbolic maximal $AdS_{3}$ spacetime with $S$ as a Cauchy surface, ($\rho_{l}$,$\rho_{r}$) the holonomy of $M$. Let $\lambda_{+}$, $\lambda_{-}$ the upper and lower pleating laminations of the convex core of $M$, $m_{+}$ and $m_{-}$ the corresponding boundary hyperbolic structures.

Then $$\begin{array}{ll}\rho_{l}= E_{\lambda_{+}}^{l}(m_{+}) \\

\rho_{r}= E_{\lambda_{+}}^{r}(m_{+}) \\

\rho_{l}= E_{\lambda_{-}}^{r}(m_{-}) \\

\rho_{r}= E_{\lambda_{-}}^{l}(m_{-})~. \end{array} $$

\begin{figure}[ht]

\begin{center}
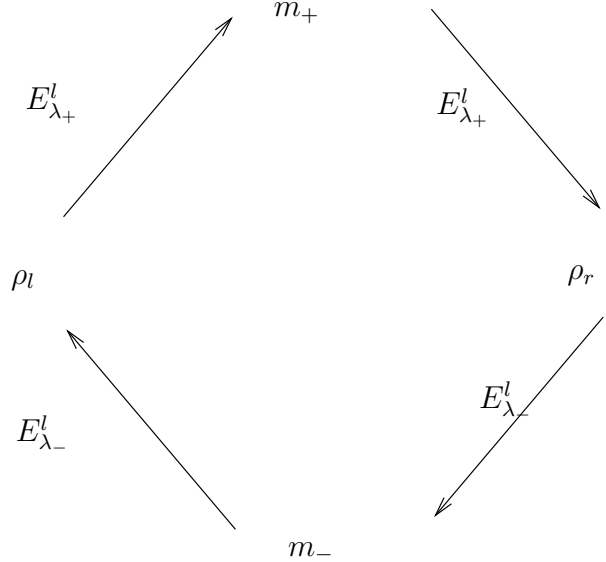
\caption{Mess diagram}
\end{center}

\end{figure}

Thanks to Thurston's theorem, our map $\Phi$ of the previous section is continuous. Via the earthquake theorem, surjectivity of $\Phi$ is thus equivalent to the following statement:

\textbf{Prescribing middle points of 2 intersecting earthquakes paths}
For any two points $m_{+}$ and $m_{-}$ in ${\mathcal T}(S)$, there are left and right earthquakes, joining two points say $\rho_{l}$ and $\rho_{r}$ in ${\mathcal T}(S)$, such that $m_{+}$ and $m_{-}$ are the middle points of the corresponding earthquake paths.

Again we fail to prove uniqueness.

\section{Basic results and strategy of the proof}

\subsection{Anti de Sitter geometry in dimension $3$ and globally hyperbolic manifolds}

Anti de Sitter $3$-space is the model space of Lorentzian $3$-manifolds of constant sectional curvature equal to $-1$. Like hyperbolic space, it has a boundary at infinity, Einstein space of dimension $2$, with a conformal Lorentzian structure, and it has several equivalent definitions (which give isometric spaces).

Apart from the quadric model we already mentioned, in dimension $3$, it is also the Lie group $SL_{2}\mathbb{R}$ with its induced Killing metric. The Lie group model allows us to identify the group of isometries of $AdS_{3}$ with $$SL_{2}\mathbb{R} \times SL_{2}\mathbb{R}/J~,$$ where $J$ is the subgroup of order $2$ generated by the matrix couple $(I_{2},-I_{2})$. Therefore this isometry group is a double cover of $PSL_{2}\mathbb{R} \times PSL_{2}\mathbb{R}$. The space of holonomies of $AdS_{3}$ spacetimes (modulo conjugation) is thus a double cover of the space of representations of $\pi_{1}(M)$ to $PSL_{2}\mathbb{R} \times PSL_{2}\mathbb{R}$ (modulo conjugation). A point in that representation space will be called a pair of left and right holonomy. 

A theorem of Mess asserts that, for a GHMC $AdS_{3}$ spacetime M, with Cauchy surface $S$, the left and right holonomies are both discrete, faithful representations which act (properly discontinuously and) cocompacty on the hyperbolic plane: these are Fuchsian representations. Conversely Mess proved that to such a pair of Fuchsian representations is the holonomy pair of a unique GHMC $AdS_{3}$ manifold $M$ homeomorphic to $S \times \mathbb{R}$ (in particular, $\pi_{1}(x,M)$ and $\pi_{1}(x,S)$, with suitable choice of basepoint $x$, are isomorphic). This is the exact analog of Bers theorem for quasifuchsian hyperbolic $3$-manifold. This gives a natural homeomorphism between the space of GHMC $AdS_{3}$ structures (with orientation and time orientation) and ${\mathcal T}(S)\times {\mathcal T}(S)$.

\subsection{How to show that $\Phi$ is onto}

Recall that $\Phi$ is the map from ${\mathcal T}(S) \times {\mathcal T}(S)$ to itself which associates to a GHMC structure on a manifold M with $\pi_{1}(M) \cong \pi_{1}(S)$ the ordered pair of upper and lower boundary metrics on its convex core. Thanks to the earthquake theorem, $\Phi$ is continuous. We just need to show that it is a proper map, and since it will have a well defined degree, that it is a map of degree one. It then follows that $\phi$ is onto.

Properness theorem:

With notations as in the previous sections, if $m_{+}$ lies in a compact set in ${\mathcal T}(S)$, $m_{+}$ and $m_{-}$ lie in the image of $\Phi$, if preimages $\rho_{l}$ and $\rho_{r}$ in ${\mathcal T}(S)$ are such that $\rho_{l}$ leaves every compact subset of ${\mathcal T}(S)$, then so does $m_{-}$.

We may rephrase it with sequences in ${\mathcal T}(S)$ instead of compact sets:

if $(m^{n}_{+})_{n \geq 0}$ is a sequence in ${\mathcal T}(S)$ converging to $m^{\infty}_{+}$ in ${\mathcal T}(S)$, $(m^{n}_{-})_{n \geq 0}$, $(\rho^{n}_{l})_{n \geq 0}$ and $(\rho^{n}_{r})_{n \geq 0}$ are sequences in ${\mathcal T}(S)$ such that $(\rho^{n}_{r})_{n \geq 0}$ tends to $\infty$, then $(m^{n}_{-})_{n \geq 0}$ also tends to $\infty$.

In our context, we can (equivalently) replace "`$(\rho^{n}_{r})_{n \geq 0}$ tends to $(\infty$" by "`$l_{m^{n}_{+}}(\lambda^{n}_{+}))$ tends to $\infty$". See next section for further details.

Degree theorem:

$\Phi$ is properly homotopic to a homeomorphism. Since a homeomorphism has degree one, it then follows that our map $\Phi$ has degree one.

\section * {Acknowledgements}

I'd like to thank Jean Marc Schlenker for his grateful help and guidance during this work. I would also like to thank F.Bonsante, T.Barbot, C.Series, S.Kerckhoff, J.Danciger for fruitful discussions.

\section{Outline of the proof of the properness theorem}

Our main result here is the following theorem:

\begin{theo}[properness of $\Phi$]
Let $\Phi$ be the map defined in the previous section.
Then $\Phi$ is a proper map.

\end{theo}

We prove this theorem with $3$ propositions, which will be proven in the next section.

The first one is a slight refinement of Lemma $4.2$ of ~\cite{MR2913100}.

Let $S$ be a closed surface of negative Euler characteristic, $\alpha_{0}>0$. Let $g \in {\mathcal T}(S)$ and let $c$ a closed geodesic for $g$. For any point $x \in c$, let $g_{x}^{r}(\alpha_{0})$ (resp. $g_{x}^{l}(\alpha_{0})$) be the geodesic segment of length $\alpha_{0}$ (for the metric $g$), orthogonal at $x$ to the right (resp. the left) of $c$. Let $n_{l}(x)$ (resp. $n_{r}(x)$) be the intersection number of $g_{x}^{r}(\alpha_{0})$ (resp. $g_{x}^{l}(\alpha_{0})$) with $c$, including $x$.

\begin{pro}
There exists a constant $l_{0}$ (depending only on $\alpha_{0}$ and the genus of $S$) such that for all $\delta_{0}>0$, there exists some $\beta_{0}>0$ (depending on $\alpha_{0}$, $\delta_{0}$ and the genus of $S$) such that:
$$l_{g}(\{x \in c : \inf(n_{l}(x), n_{r}(x)) \leq \beta_{0}l_{g}(c)\}) \leq \delta_{0}l_{g}(c) + l_{0}.$$

\end{pro}

What the proposition says is: if the length of $c$ is large enough, for most points $x$ of $c$, the left (resp. right) going arc orthogonal to $c$ at $x$, of fixed length $\alpha_{0}$ intersects $c$ a lot.

By density of weighted multicurves in ${\mathcal ML}(S)$ and continuity of length functions ~\cite{MR787659} and intersection numbers, we get a similar version for more general measured laminations on $S$, namely:

\begin{cor}
With the same notations as above, let $\lambda \in {\mathcal ML}(S)$ be a measured lamination whose support is a simple closed curve $c$. Then there exists a constant $l_{0}$ (depending only on $\alpha_{0}$ and the genus of $S$) such that for all $\delta_{0}>0$, there exists some $\beta_{0}>0$ (depending on $\alpha_{0}$, $\delta_{0}$ and the genus of $S$) such that:
$$l_{g}(\{x \in c : \inf(i(g^{l}_{x}(\alpha_{0}),\lambda), i(g^{r}_{x}(\alpha_{0}),\lambda)) \leq \beta_{0}l_{g}(\lambda)\}) \leq \delta_{0}l_{g}(\lambda) + l_{0}.$$
\end{cor}

Recall that $l_{g}(\lambda)= wl_{g}(c)$ if $\lambda$ has weight $w$ and support the simple closed curve $c$~(see ~\cite{MR787659}).

The next proposition says that for points of $c$ (on $\partial_{+}C$) satisfying the reverse inequality (i.e. $\inf(n_{l}(x), n_{r}(x)) > \beta_{0}l_{g}(c)\}$), the distance to the opposite boundary component of the convex core is near $\pi/2$. Cutting $c$ by an orthogonal (totally geodesic) plane at such points, the assertion is about $2$-dimensional AdS geometry.

Let $C$ be a convex subset of $AdS_{2}$, with spacelike boundary, whose frontier in $AdS_{2} \cup \partial_{\infty}AdS_{2}$ consists of $2$ points, one on each component of $\partial_{\infty}AdS_{2}$. Let $\alpha_{0}>0$, $x$ be a point of the upper boundary component $\partial_{+}C$ of $C$, $\Pi_{x}$ a support line of $C$ at $x$, $P_{l,x}$ (resp. $P_{r,x}$) a support line of $C$ at the point of $\partial_{+}C$ at distance $\alpha_{0}$ from the right (resp. the left) of $x$. We denote by $\phi_{r,x}$ (resp. $\phi_{l,x}$) the angle between $\Pi_{x}$ and $P_{r,x}$ (resp. $P_{l,x}$). Let $\tau$ be a time-like geodesic orthogonal at $x$ to $\Pi_{x}$.

\begin{pro}
For all $\epsilon>0$, there exist $A>0$, $\eta_{0} > 0$ such that if $$\inf(\phi_{r,x},\phi_{l,x}) \geq A$$ and $$\alpha_{0} \leq \eta_{0}$$ then $$l(\tau \cap C)\geq \pi/2-\epsilon.$$
\end{pro}

Note that it is proved in ~\cite{MR2913100} that we always have $l(\tau \cap C) < \pi/2$.

The next and last proposition applied in our situation enables us to compare the length of $\lambda_{+}$ (in case its support is a simple closed curve $c$) for $m_{+}$ and $m_{-}$, considering a totally geodesic time-like annulus containing $c$ which is orthogonal, along $c$, to a support plane at some point of $c$.

Let $M$, $m_{+}$,$m_{-}$, $\lambda_{+}$,$\lambda_{-}$ be as in Theorem 1.

Suppose that the support of $\lambda_{+}$ is a closed curve $c$. Let $c_{-}$ be the curve obtained by intersecting the lower boundary component of the convex core of M with the totally geodesic timelike annulus orthogonal to a support plane at some point of $c$. Let $K$ be a compact subset of ${\mathcal T}(S)$.

\begin{pro}
For all $\epsilon > 0$, there exists some $A > 0$ such that if $m_{+} \in K$ and $l_{m_{+}}(\lambda_{+})\geq A$ then $l_{m_{-}}(\lambda_{+})\leq \epsilon l_{m_{+}}(\lambda_{+})$.
\end{pro}

This last proposition extends to arbitrary $\lambda_{+}$ by density of weighted multicurves is ${\mathcal ML}(S)$, by continuity of length functions and continuity of $m_{-}$ with respect to $m_{+}$ and $\lambda_{+}$.\\
Let us explain why the 3 propositions imply the theorem. In fact, the compactness theorem is a consequence of Proposition $5$, which is a result of both Proposition $2$ and $4$.

Suppose there exist a sequence $m^{n}_{+}$ in ${\mathcal T}(S)$ converging to $m^{\infty}_{+} \in {\mathcal T}(S)$ and a sequence $\rho^{n}_{l}$ leaving any compact set in ${\mathcal T}(S)$. Then by Mess' diagram, the corresponding sequence of upper boundary's pleating lamination $\lambda^{n}_{+}$ leaves any compact set in ${\mathcal ML}(S)$. In particular, its length, measured with respect to $m^{n}_{+}$, tends to infinity.

Let $\epsilon>0$ and A be as in Proposition 5. There exists $n_{0}$ (depending on $\epsilon$ and $A$) such that for $n\geq n_{0}$, one has $l_{m^{n}_{+}}(\lambda^{n}_{+}) \geq A$, hence $l_{m^{n}_{-}}(\lambda^{n}_{+}) \leq \epsilon l_{m^{n}_{+}}(\lambda^{n}_{+})$. Since $m^{n}_{+}$ converges in ${\mathcal T}(S)$, this implies that $m^{n}_{-}$ leaves every compact subset of ${\mathcal T}(S)$ (i.e tends to infinity), proving the properness of $\Phi$.

\section{Proof of the main propositions}

We begin with the proof of Proposition $2$, we need the following lemma.

Let $S$ be a closed surface of negative Euler characteristic, $g \in {\mathcal T}(S)$, $\alpha_{0}>0$. Let $c$ a closed geodesic for $g$. For any point $x \in c$, let $g_{x}^{r}(\alpha_{0})$ (resp. $g_{x}^{l}(\alpha_{0})$) be the geodesic segment of length $\alpha_{0}$ (for the metric $g$), orthogonal at $x$ to the right (resp. the left) of $c$. Let $n_{l}(x)$ (resp. $n_{r}(x)$) be the intersection number of $g_{x}^{r}(\alpha_{0})$ (resp. $g_{x}^{l}(\alpha_{0})$) with $c$, including $x$.

\begin{lem}
There exists a constant $l_{0}$ (depending only on $\alpha_{0}$ and the genus of $S$) such that for all $\delta_{0}>0$, there exists some $\beta_{0}>0$ (depending on $\alpha_{0}$, $\delta_{0}$ and the genus of $S$) such that:
$$l_{g}(\{x \in c : n_{l}(x) \leq \beta_{0}l_{g}(c)\}) \leq \delta_{0}l_{g}(c) + l_{0}.$$

\end{lem}

The similar statement is true for $n_{r}(x)$ by the same argument. The proposition then follows by combining both statements.

\begin{sub}
There exists $\gamma_{0}$ (depending on $\alpha_{0}$) as follows. Let $D_{1}$, $D_{2}$ be two disjoint lines in the hyperbolic plane and let $x$ be a point in the connected component of the complement in $H^2$ of $D_{0} \cup D_{1}$ whose boundary contains those two lines. Suppose that $d(x,D_{0}) \leq \gamma_{0}$ and $d(x,D_{1}) \leq \gamma_{0}$. Then the geodesic segment of length $\alpha_{0}$ starting orthogonally from $D_{0}$ and containing $x$ intersects $D_{1}$.
\end{sub}

\begin{figure}[ht]

\begin{center}
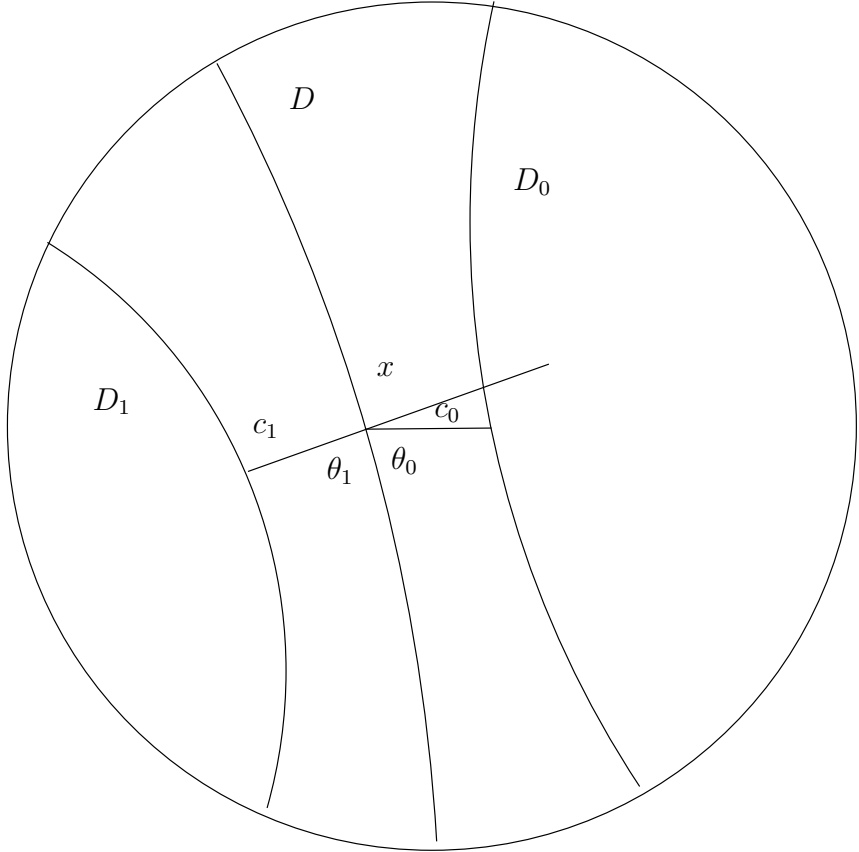
\caption{intersections de g\'eod\'esiques hyperboliques planes}
\end{center}

\end{figure}
\begin{proof}[\textbf{Proof of the sublemma}]
Let $c_{0}$ and $c_{1}$ be the geodesic segments joining $x$ to its orthogonal projections on $D_{0}$ and $D_{1}$, respectively. Then $d(x,D_{0})=l(c_{0})$ and $d(x,D_{1})=l(c_{1})$. Let $D$ be a hyperbolic line containing $x$ and disjoint from $D_{0}$ and $D_{1}$, $\theta_{0}$ and $\theta_{1}$ be the respective angles between $D$ and $c_{0}$, $D$ and $c_{1}$. Then $$\cosh(l(c_{0}))\sin(\theta_{0})=\cosh(d(D,D_{0}))$$ and $$\cosh(l(c_{1}))\sin(\theta_{1})=\cosh(d(D,D_{1}))$$
(See ~\cite[page 88]{MR2249478}). Since $d(D,D_{0}) \leq l(c_{0}) \leq \gamma_{0}$ and $d(D,D_{1}) \leq l(c_{1}) \leq \gamma_{0}$, $$ |\pi/2 - \theta_{0}| \leq \delta_{0} , |\pi/2 - \theta_{1}| \leq \delta_{0} $$ for sufficiently small $\gamma_{0}$ (depending on $\delta_{0}$).

Let now $c'_{1}$ be the half-line extending $c_{1}$ on the other side of $x$. The angle between $c_{0}$ and $c'_{1}$ is less than $2\delta_{0}$. If $\gamma_{0}$ is small enough, then $\delta_{0}$ is small and $c'_{1}$ intersects $D_{0}$ (which is orthogonal to $c_{0}$) at distance less than $\alpha_{0}$. This is the required statement.
\end{proof}

\begin{proof}[\textbf{Proof of the lemma}]
Let $\delta_{0}>0$, $\gamma_{0}$ as in the sublemma, and $0<\beta_{0}<\delta_{0}\gamma_{0}/(2\pi|\chi(S)|)$. Let $\overline{c}=\{x \in c : n_{l}(x) \leq \beta_{0}l_{g}(c)\}$.
Fix $\gamma_{0}$ as in the sublemma and consider the normal exponential map: $$ exp: \left.\begin{array}{ll}\overline{c} \times [0,\gamma_{0}] \rightarrow S \\ (s,r) \mapsto g_{s}^{l}(r)\end{array}\right. $$
It is a distance increasing map so it increases areas. Moreover, the sublemma shows that each $x \in S$ has at most $n_{0}+1$ preimages in $\overline{c} \times [0,\gamma_{0}]$, where $n_{0}$ is the integer part of $\beta_{0}l_{g}(c)$. Indeed, suppose that $x$ is the image of both $(y,r)$ and $(y',r')$ and let $\overline{x}$,$\overline{y}$ and $\overline{y}'$ be lifts of $x$,$y$ and $y'$ to the universal cover $H^2$ of $(S,g)$ chosen so that the corresponding lift of $\exp(\{y\}\times [0,r])$ (resp. $\exp(\{y'\}\times [0,r'])$) contains both $\overline{x}$ and $\overline{y}$ (resp. $\overline{x}$ and $\overline{y}$'). Then one of the following holds:

$$\left.\begin{array}{ll} $$exp(\{y\}\times [0,r])$ intersects the lift of $c$ containing $y'$$ \\ $$exp(\{y'\}\times [0,r'])$ intersects the lift of $c$ containing $y$$ \\ $$exp(\{y\}\times [0,\alpha_{0}])$ intersects the lift of $c$ containing $y'$ and conversely.$ \end{array}\right.$$

(The $3rd$ possibility is obtained by applying the SubLemma) This argument shows that there are at most $n_{0} + 1$ preimages of $x$ in $\overline{c}\times [0,\alpha_{0}]$.

Since the area of $(S,g)$ is $2\pi|\chi(S)|$, it follows that $$\gamma_{0}l_{g}(\overline{c}) \leq 2\pi(n_{0}+1)|\chi(S)| \leq 2\pi(\beta_{0}l_{g}(c)+1)|\chi(S)| \leq \gamma_{0}(\delta_{0}l_{g}(c) + l_{0})$$

if one puts $l_{0}=2\pi|\chi(S)|/\gamma_{0}$.
\end{proof}

We now prove Proposition $4$, we first need a lemma comparing the lengths of two past (resp. future) convex spacelike arcs, with the same endpoints, one being in the past (resp. future) of the other.

\begin{lem}[lengths of convex arcs with fixed endpoints]
Let $\sigma_{0}$ and $\sigma_{1}$ be two future (resp. past) convex spacelike curves with the same endpoints in $AdS_{2}$. Suppose that $\sigma_{0}$ lies in the future (resp. past) of $\sigma_{1}$. Then $$l(\sigma_{0}) \geq l(\sigma_{1}).$$
\end{lem}

\begin{proof}[\textbf{Proof of the lemma}]
We first examine the case where $\sigma_{0}$ and $\sigma_{1}$ are piecewise geodesic arcs.

First, suppose $\sigma_{0}$ is a geodesic. Then by induction on the number of geodesic arcs in $\sigma_{1}$, this case is a consequence of Sublemma $4.6$ of ~\cite{MR2913100} (the "reverse triangle inequality").

Now suppose $\sigma_{0}$ is a piecewise geodesic arc. Consider the geodesic arc $\sigma'_{0}$ joining the endpoints of $\sigma_{0}$. By orthogonal projection of each vertex of $\sigma_{1}$ on $\sigma'_{0}$, we get timelike arcs joining those vertices to points of $\sigma'_{0}$ and orthogonal to $\sigma'_{0}$. Now we can pull back $\sigma_{1}$ along those arcs (with the normal exponential map) until it reaches (at least) one point of $\sigma_{0}$. We get a piecewise geodesic arc $\sigma'_{1}$, which is longer than $\sigma_{1}$, spacelike and future convex but this time $\sigma'_{1}$ and $\sigma_{0}$ have a common point. We can then apply the induction argument on both sides of this common point on $\sigma_{0}$ and $\sigma'_{1}$.

Then the general case follows by approximation.
\end{proof}

We keep the same notations as in the statement of Proposition $4$ .

\begin{proof}[\textbf{Proof of Proposition $4$}]
Letting $x_{l}$ (resp. $x_{r}$) be the intersection point of $\Pi_{x}$ and $P_{l,x}$ (resp. $P_{r,x}$), we note that the (spacelike) distance between $x$ and $x_{l}$ (resp. $x$ and $x_{r}$) is less than $\alpha_{0}$, by Lemma $1$.

\begin{figure}[ht]

\begin{center}
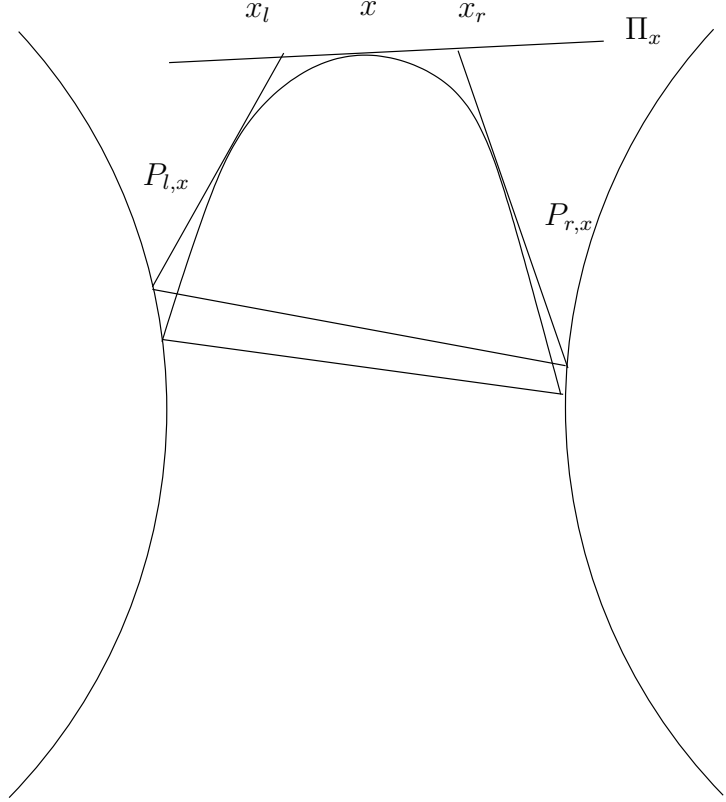
\caption{cutting the convex core by a plane orthogonal to a support line}
\end{center}

\end{figure}

The configuration made by the three spacelike lines and the spacelike line $\Pi^{'}$ joining the two points at infinity of $P_{l,x}$ and $P_{r,x}$ in the past of $\Pi_{x}$ admits a well-known limit situation: when $\alpha_{0}$ goes to zero, $x_{l}$ and $x_{r}$ tend to $x$, moreover when $\phi_{l,x}$ and $\phi_{r,x}$ tend to infinity, $P_{l,x}$ and $P_{r,x}$ tend to lightlike lines so as to let $\Pi^{'}$ tend to the dual line $\Pi^{*}$ to $x$, at (timelike) distance $\pi/2$ from $x$.

The result follows since by convexity, $$l(\tau \cap C)\geq l(I^{+}(\Pi^{'})\cap \tau)$$
\end{proof}

\begin{figure}[ht]

\begin{center}
\includegraphics{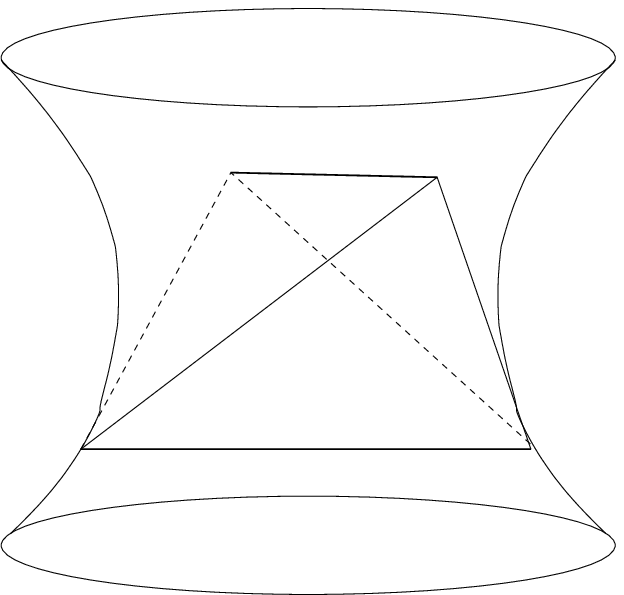}
\caption{dual lines}
\end{center}

\end{figure}

Note that in our case, since $\alpha_{0}$ tends to $0$, having our intersection numbers $n_{l}(x)$ and $n_{r}(x)$ tend to $\infty$ is equivalent to having the angles $\phi_{l,x}$ and $\phi_{r,x}$ tend to $\infty$ (since the left and right geodesic segments of length $\alpha_{0}$ at $x$ meet $c$ with an angle near from $\pi/2$ )

It remains to prove Proposition $5$ . We keep the same notations as in the statement of this proposition.

\begin{proof}[\textbf{Proof of Proposition $5$}]
Let $\epsilon > 0$, let $K$ be a compact subset of ${\mathcal T}(S)$. We choose $\epsilon' > 0$ (say $\epsilon' < \pi/2$), such that $$\cos(\pi/2-\epsilon') \leq \epsilon/2.$$ Then there are constants $A'$ and $\eta_{0}$ such that if $$\inf(\phi_{r,x},\phi_{l,x}) \geq A'$$ and $$\alpha_{0} \leq \eta_{0}$$ then $$l(\tau \cap C)\geq \pi/2-\epsilon'$$ with the notations of Proposition 4 .

Let $A(c)$ be a totally geodesic timelike annulus orthogonal to a support plane of $C$ along $c$. Let $c_{-}=A(c)\cap \partial_{-}C$ and let $c'_{-}$ be the curve at distance $\pi/2-\epsilon'$ in the past of $c$ (lying in $A(c)$). Then $c'_{-}$ is a future convex curve and $$l(c'_{-}) = \cos(\pi/2-\epsilon')l_{m_{+}}(c),$$ hence $$l(c'_{-})\leq (\epsilon/2)l_{m_{+}}(c).$$

Consider now the orthogonal projection $Pr_{-}$ from $c_{-}$ to $c'_{-}$. There are two types of points on $c_{-}$, those at distance greater than or equal to $\pi/2 - \epsilon'$ from $c$ and the others.

\begin{figure}[ht]

\begin{center}
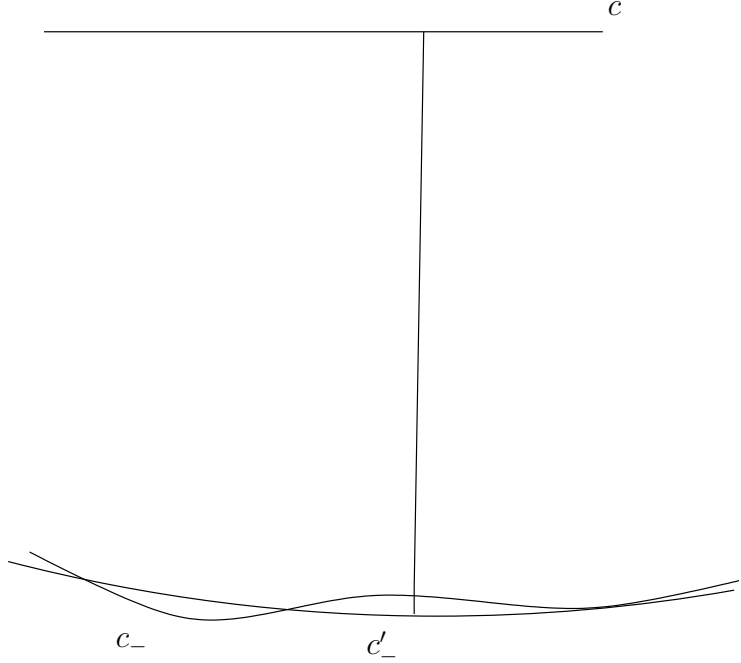
\caption{comparing length on top and bottom of the convex core}
\end{center}

\end{figure}

Let $\beta_{0}$ as in Proposition $2$ and $\delta_{0}=\epsilon/4$.
Proposition 4 gives the existence of constants $A'$ and $\eta_{0}$ (where $\alpha_{0} \leq \inf (\eta_{0}, \epsilon/2)$), we choose $c$ such that $l_{m_{+}}(c) \geq A'/\beta_{0}$.  
Let $$B_{-}=\{\tau_{x}\cap \partial_{-}C,x \in c, \inf(n_{l}(x),n_{r}(x)) \geq \beta_{0}l(c) \},$$ then points of $B_{-}$ belongs to $$\{ y \in c_{-}, d(y,c)\geq \pi/2-\epsilon'\}.$$ The points of the first type are in the past of $c'_{-}$, so by Lemma 1, $$l(B_{-}) \leq  l(Pr_{-}(B_{-})),$$ hence by composition of projections, $$l(B_{-}) \leq (\epsilon/2)l(Pr(B_{-})), $$ where $Pr$ is the projection from $c_{-}$ to $c$.\\
Let $$B= \{\tau_{x}\cap \partial_{-}C,x \in c, \inf(n_{l}(x),n_{r}(x)) \leq \beta_{0}l(c) \}.$$

Lemma 1 gives $$\left.\begin{array}{ll}l(\{x \in c, \inf(n_{l}(x),n_{r}(x)) \leq \beta_{0}l(c)\}) \geq l(B)\end{array}\right.$$ 

Then, by choosing $A \geq \sup( A'/\beta_{0},4l_{0}/\epsilon)$ (where $l_{0}$ comes from Proposition 4), $$\left.\begin{array}{ll}l(\{x \in c, \inf(n_{l}(x),n_{r}(x)) \leq \beta_{0}l(c)\}) \leq \delta_{0}l(c) + l_{0}\leq \epsilon l(c)/4 + \epsilon l(c)/4\end{array} \right.$$ hence $$\left.\begin{array}{ll}\epsilon l(c)/2 \geq l(B)\end{array}\right. .$$\\ We get the desired result by addition, then by passing to the limit for a general lamination.
\end{proof}

\section{Computing degrees}

In this section, we investigate the degree of the proper map $\Phi$. Our main aim is the

\begin{pro}[degree one theorem]
Let $\Phi$ be the map defined in section 2. Then $\Phi$ has degree one.

\end{pro}

Recall that the degree of proper maps between manifolds is well defined and that a map whose degree is nonzero is surjective. Thus the last proposition ends the proof of the main theorem of the article.

In ~\cite{MR2895066}, they prove the existence of a map $\Phi_{k_{+},k_{-}}$ from ${\mathcal GH}(S)$ (GHMC  $AdS_{3}$ structures with Cauchy surface homeomorphic to $S$, identified with ${\mathcal T}(S) \times {\mathcal T}(S)$) to ${\mathcal T}(S) \times {\mathcal T}(S)$ which associates to a GHMC AdS3 manifold the pair of conformal structures of the unique past convex (resp. future convex) surfaces with constant sectional curvatures equal to $k_{+}$ and $k_{-}$ respectively (see also ~\cite{cyclic} where this map is considered).

They also prove that the map $\varphi$ $$(k^{+},k^{-},\rho_{l},\rho_{r}) \mapsto  \Phi_{k_{+},k_{-}}(\rho_{l},\rho_{r}) = (\Phi^{+}_{k_{+},k_{-}}(\rho_{l},\rho_{r}),\Phi^{-}_{k_{+},k_{-}}(\rho_{l},\rho_{r})) $$ is continuous (lemma 12.4) on $(-\infty,-1)^{2}\times {\mathcal T}(S) \times {\mathcal T}(S)$. 

The continuity on $(-\infty,-1]^{2}\times {\mathcal T}(S) \times {\mathcal T}(S)$ follows from the following lemma:

\begin{lem}[convergence of a sequence of hyperbolic metrics]
Let $m$ be a point of ${\mathcal T}(S)$ and $C_{n}$ be a sequence of real numbers greater than one which converge to $1$.

Then any sequence $m_{n} \in {\mathcal T}(S)$ satisfying $$m_{n} \leq C_{n}m $$ converge to m.
\end{lem}

This is a consequence of a classical fact proved by Thurston ~\cite{1998math......1039T}. Indeed, by the following lemma, such a sequence admits a convergent subsequence, whose limit $m_{\infty}$ satisfies $m_{\infty} \leq m$. Then Thurston proved that this gives $m_{\infty}= m$.

The map $\varphi$, when restricted to any set of the form $[-C,-1]^{2}\times {\mathcal T}(S) \times {\mathcal T}(S)$ (for any $C > 1$), is proper thanks to the following lemmas (which are corollaries of results from ~\cite{1998math......1039T}):

\begin{lem}[compactness of sets of metrics first version]
Let $m$ be a point of ${\mathcal T}(S)$ and let $C>1$.

Then the set of metrics $m'$ in ${\mathcal T}(S)$ such that: $Cm'\geq m$ is compact in ${\mathcal T}(S)$.
\end{lem}
The same lemma is true when considering metrics $m'$ such that $m'\leq Cm$.

By a slight extention we get the following lemma:

\begin{lem}[compactness of sets of metrics revisited]
Let $K$ be a compact of ${\mathcal T}(S)$, let $C>1$. Then the set of metrics $m'$ such that $Cm'>m$ for some $m \in K$ is compact.
\end{lem}

Consider the slicing of an $AdS$ GHMC manifold given by the map $\varphi$ . For any $k< -1$, let $k^{*}$ be the curvature of the dual to the surface of curvature $k$. Then we have:(formula for $k*$)

Let us then prove the properness of the map $\varphi$. Let $(m_{n}^{+})$ and $(m_{n}^{-})$ be convergent sequences in ${\mathcal T}(S)$ and $\rho_{n}^{l}$, $\rho_{n}^{r}$, $k_{n}^{+}$ and $k_{n}^{-}$ be sequences such that the image of $$(k_{n}^{+},k_{n}^{-},\rho_{n}^{l},\rho_{n}^{r})$$ by $\varphi$ is $((m_{n}^{+}),(m_{n}^{-}))$.

By properness of $\Phi_{-1,-1} = \Phi$ (proved in the previous section) and the fact that $$ C\Phi^{+}_{k_{+},k_{-}} < \Phi^{+}_{-1,-1}$$ and $$ C\Phi^{-}_{k_{+},k_{-}} < \Phi^{-}_{-1,-1}, $$ the sequences $k_{n}^{+}$,$k_{n}^{-}$,$\rho_{n}^{l}$,$\rho_{n}^{r}$ stay in a compact set, hence they admit convergent subsequences.

By invariance of the degree of a map under a proper homotopy, all the maps $\Phi_{k_{+},k_{-}}$ have the same degree, which is given by this last lemma:

\begin{lem}[degree of $\Phi_{k,k^{*}}$]
The map $\Phi_{k,k^{*}}$ has degree one
\end{lem}

In fact it is a homeomorphism.

\bibliography{biblio2}
\bibliographystyle{plain}
\end{document}

%% file: Mess-diagram2.pstex_t
\begin{picture}(0,0)%
\includegraphics{Mess-diagram2.ps}%
\end{picture}%
%
%
\setlength{\unitlength}{3947sp}%
\begingroup\makeatletter\ifx\SetFigFont\undefined%
\gdef\SetFigFont#1#2#3#4#5{%
  \reset@font\fontsize{#1}{#2pt}%
  \fontfamily{#3}\fontseries{#4}\fontshape{#5}%
  \selectfont}%
\fi\endgroup%
\begin{picture}(3747,3637)(1246,-10125)
\put(2911,-6661){\makebox(0,0)[lb]{\smash{{\SetFigFont{12}{14.4}{\rmdefault}{\mddefault}{\updefault}{\color[rgb]{0,0,0}$m_{+}$}%
}}}}
\put(3001,-10051){\makebox(0,0)[lb]{\smash{{\SetFigFont{12}{14.4}{\rmdefault}{\mddefault}{\updefault}{\color[rgb]{0,0,0}$m_{-}$}%
}}}}
\put(1261,-8341){\makebox(0,0)[lb]{\smash{{\SetFigFont{12}{14.4}{\rmdefault}{\mddefault}{\updefault}{\color[rgb]{0,0,0}$\rho_{l}$}%
}}}}
\put(4756,-8311){\makebox(0,0)[lb]{\smash{{\SetFigFont{12}{14.4}{\rmdefault}{\mddefault}{\updefault}{\color[rgb]{0,0,0}$\rho_{r}$}%
}}}}
\put(1291,-9331){\makebox(0,0)[lb]{\smash{{\SetFigFont{12}{14.4}{\rmdefault}{\mddefault}{\updefault}{\color[rgb]{0,0,0}$E^{l}_{\lambda_{-}}$}%
}}}}
\put(4201,-9121){\makebox(0,0)[lb]{\smash{{\SetFigFont{12}{14.4}{\rmdefault}{\mddefault}{\updefault}{\color[rgb]{0,0,0}$E^{l}_{\lambda_{-}}$}%
}}}}
\put(1351,-7246){\makebox(0,0)[lb]{\smash{{\SetFigFont{12}{14.4}{\rmdefault}{\mddefault}{\updefault}{\color[rgb]{0,0,0}$E^{l}_{\lambda_{+}}$}%
}}}}
\put(3931,-7276){\makebox(0,0)[lb]{\smash{{\SetFigFont{12}{14.4}{\rmdefault}{\mddefault}{\updefault}{\color[rgb]{0,0,0}$E^{l}_{\lambda_{+}}$}%
}}}}
\end{picture}%

%% file: intersection1.pstex_t
\begin{picture}(0,0)%
\includegraphics{intersection1.ps}%
\end{picture}%
%
%
\setlength{\unitlength}{3947sp}%
\begingroup\makeatletter\ifx\SetFigFont\undefined%
\gdef\SetFigFont#1#2#3#4#5{%
  \reset@font\fontsize{#1}{#2pt}%
  \fontfamily{#3}\fontseries{#4}\fontshape{#5}%
  \selectfont}%
\fi\endgroup%
\begin{picture}(5348,5351)(822,-6664)
\put(1366,-3886){\makebox(0,0)[lb]{\smash{{\SetFigFont{12}{14.4}{\rmdefault}{\mddefault}{\updefault}{\color[rgb]{0,0,0}$D_{1}$}%
}}}}
\put(4006,-2506){\makebox(0,0)[lb]{\smash{{\SetFigFont{12}{14.4}{\rmdefault}{\mddefault}{\updefault}{\color[rgb]{0,0,0}$D_{0}$}%
}}}}
\put(2371,-4036){\makebox(0,0)[lb]{\smash{{\SetFigFont{12}{14.4}{\rmdefault}{\mddefault}{\updefault}{\color[rgb]{0,0,0}$c_{1}$}%
}}}}
\put(3511,-3931){\makebox(0,0)[lb]{\smash{{\SetFigFont{12}{14.4}{\rmdefault}{\mddefault}{\updefault}{\color[rgb]{0,0,0}$c_{0}$}%
}}}}
\put(2836,-4321){\makebox(0,0)[lb]{\smash{{\SetFigFont{12}{14.4}{\rmdefault}{\mddefault}{\updefault}{\color[rgb]{0,0,0}$\theta_{1}$}%
}}}}
\put(3241,-4261){\makebox(0,0)[lb]{\smash{{\SetFigFont{12}{14.4}{\rmdefault}{\mddefault}{\updefault}{\color[rgb]{0,0,0}$\theta_{0}$}%
}}}}
\put(2596,-1996){\makebox(0,0)[lb]{\smash{{\SetFigFont{12}{14.4}{\rmdefault}{\mddefault}{\updefault}{\color[rgb]{0,0,0}$D$}%
}}}}
\put(3151,-3676){\makebox(0,0)[lb]{\smash{{\SetFigFont{12}{14.4}{\rmdefault}{\mddefault}{\updefault}{\color[rgb]{0,0,0}$x$}%
}}}}
\end{picture}%

%% file: support.pstex_t
\begin{picture}(0,0)%
\includegraphics{support.ps}%
\end{picture}%
%
%
\setlength{\unitlength}{3947sp}%
\begingroup\makeatletter\ifx\SetFigFont\undefined%
\gdef\SetFigFont#1#2#3#4#5{%
  \reset@font\fontsize{#1}{#2pt}%
  \fontfamily{#3}\fontseries{#4}\fontshape{#5}%
  \selectfont}%
\fi\endgroup%
\begin{picture}(4501,5101)(1013,-7629)
\put(3226,-2701){\makebox(0,0)[lb]{\smash{{\SetFigFont{12}{14.4}{\rmdefault}{\mddefault}{\updefault}{\color[rgb]{0,0,0}$x$}%
}}}}
\put(2506,-2716){\makebox(0,0)[lb]{\smash{{\SetFigFont{12}{14.4}{\rmdefault}{\mddefault}{\updefault}{\color[rgb]{0,0,0}$x_{l}$}%
}}}}
\put(3841,-2716){\makebox(0,0)[lb]{\smash{{\SetFigFont{12}{14.4}{\rmdefault}{\mddefault}{\updefault}{\color[rgb]{0,0,0}$x_{r}$}%
}}}}
\put(4891,-2851){\makebox(0,0)[lb]{\smash{{\SetFigFont{12}{14.4}{\rmdefault}{\mddefault}{\updefault}{\color[rgb]{0,0,0}$\Pi_{x}$}%
}}}}
\put(4381,-4021){\makebox(0,0)[lb]{\smash{{\SetFigFont{12}{14.4}{\rmdefault}{\mddefault}{\updefault}{\color[rgb]{0,0,0}$P_{r,x}$}%
}}}}
\put(1861,-3766){\makebox(0,0)[lb]{\smash{{\SetFigFont{12}{14.4}{\rmdefault}{\mddefault}{\updefault}{\color[rgb]{0,0,0}$P_{l,x}$}%
}}}}
\end{picture}%

%% file: projection1.pstex_t
\begin{picture}(0,0)%
\includegraphics{projection1.ps}%
\end{picture}%
%
%
\setlength{\unitlength}{3947sp}%
\begingroup\makeatletter\ifx\SetFigFont\undefined%
\gdef\SetFigFont#1#2#3#4#5{%
  \reset@font\fontsize{#1}{#2pt}%
  \fontfamily{#3}\fontseries{#4}\fontshape{#5}%
  \selectfont}%
\fi\endgroup%
\begin{picture}(4610,4267)(473,-5805)
\put(4246,-1711){\makebox(0,0)[lb]{\smash{{\SetFigFont{12}{14.4}{\rmdefault}{\mddefault}{\updefault}{\color[rgb]{0,0,0}$c$}%
}}}}
\put(1156,-5686){\makebox(0,0)[lb]{\smash{{\SetFigFont{12}{14.4}{\rmdefault}{\mddefault}{\updefault}{\color[rgb]{0,0,0}$c_{-}$}%
}}}}
\put(2731,-5731){\makebox(0,0)[lb]{\smash{{\SetFigFont{12}{14.4}{\rmdefault}{\mddefault}{\updefault}{\color[rgb]{0,0,0}$c'_{-}$}%
}}}}
\end{picture}%